\documentclass[11pt,a4paper]{article}

\usepackage[leqno]{amsmath}
\usepackage{amssymb,amsthm,upref,amscd}
\usepackage[T1]{fontenc}
\usepackage{times}
\usepackage{cite}
\usepackage{amsfonts}
\usepackage[colorinlistoftodos]{todonotes}

\usepackage{color}
\usepackage{hyperref}
\usepackage{graphicx}
\usepackage{wrapfig}

\usepackage[usenames,dvipsnames]{pstricks}
 \usepackage{epsfig}

\newtheorem{thm}{Theorem}[section]
\newtheorem{lem}[thm]{Lemma}

\newtheorem{cor}[thm]{Corollary}

\theoremstyle{definition}

\theoremstyle{remark}

\theoremstyle{remark}
\newtheorem{remark}[thm]{Remark}
\numberwithin{equation}{section}

\newcommand{\R}{{\mathbb R}}

\definecolor{blu}{rgb}{0,0,1}

\newcommand{\al}{\alpha}
\newcommand{\be}{\beta}
\newcommand{\ga}{\gamma}

\newcommand{\la}{\lambda}

\newcommand{\De}{\Delta}
\newcommand{\Ga}{\Gamma}

\newcommand{\eps}{\varepsilon}
\newcommand{\pa}{\partial}

\newcommand{\cC}{{\mathcal C}}

\newcommand{\cP}{{\mathcal P}}

\newcommand{\cS}{{\mathcal S}}

\newcommand{\weakto}{\rightharpoonup}

\newcommand{\beq[1]}{\begin{equation}\label{eq:#1}}
\newcommand{\eeq}{\end{equation}}

\begin{document}
\title{Normalized solutions for nonlinear Schr\"odinger systems}

\author{\sc{Thomas Bartsch}
\and
\sc{Louis Jeanjean}}

\date{}
\maketitle

\begin{abstract}
We consider the existence of \emph{normalized} solutions in $H^1(\R^N) \times H^1(\R^N)$ for systems of nonlinear Schr\"odinger equations which appear in models for binary mixtures of ultracold quantum gases. Making a solitary wave ansatz one is led to coupled systems of elliptic equations of the form
\[
\left\{
\begin{aligned}
-\De u_1 &= \la_1u_1 + f_1(u_1)+\pa_1F(u_1,u_2),\\
-\De u_2 &= \la_2u_2 + f_2(u_2)+\pa_2F(u_1,u_2),\\
u_1,u_2&\in H^1(\R^N),\ N\ge2,
\end{aligned}
\right.
\]
and we are looking for solutions satisfying
\[
\int_{\R^N}|u_1|^2 = a_1,\quad \int_{\R^N}|u_2|^2 = a_2
\]
where $a_1>0$ and $a_2>0$ are prescribed. In the system $\la_1$ and $\la_2$ are unknown and will appear as Lagrange multipliers. We treat the case of homogeneous nonlinearities, i.e.\ $f_i(u_i)=\mu_i|u_i|^{p_i-1}u_i$, $F(u_1,u_2)=\be|u_1|^{r_1}|u_2|^{r_2}$, with positive constants $\be, \mu_i, p_i, r_i$. The exponents are Sobolev subcritical but may be $L^2$-supercritical: $p_1,p_2,r_1+r_2\in]2,2^*[\,\setminus\left\{2+\frac4N\right\}$.
\end{abstract}




\section{Introduction}\label{sec:intro}

Elliptic systems of the form
\beq[1.0]
\left\{
\begin{aligned}
-\De u_1 &= \la_1u_1 + f_1(u_1)+\pa_1F(u_1,u_2)\\
-\De u_2 &= \la_2u_2 + f_2(u_2)+\pa_2F(u_1,u_2)\\
u_1,u_2&\in H^1(\R^N)
\end{aligned}
\right.
\eeq
have been investigated in the last decades by many authors. Surprisingly little is known about the existence of normalized solutions, i.e.\ solutions that satisfy the constraint
\beq[1.2]
\int_{\R^N} |u_1|^2 = a_1\quad\text{and}\quad \int_{\R^N}| u_2|^2 = a_2
\eeq
with $a_1,a_2>0$ prescribed. One motivation to look for normalized solutions of \eqref{eq:1.0} are coupled systems of nonlinear Schr\"odinger equations
\begin{equation}\label{syst schrod}
\begin{cases}
- i \pa_t \Psi_1 = \De \Psi_1 + g_1(|\Psi_1|)\Psi_1 + \pa_1G(|\Psi_1|^2,|\Psi_2|^2)\Psi_1\\
- i \pa_t \Psi_2 = \De \Psi_2 + g_2(|\Psi_2|)\Psi_2 + \pa_2G(|\Psi_1|^2,|\Psi_2|^2)\Psi_2
\end{cases} \text{in $\R \times \R^N$}.
\end{equation}
Since the \emph{masses}
\[
\int_{\R^N} |\Psi_1|^2 \quad \text{and} \quad \int_{\R^N} |\Psi_2|^2
\]
are preserved along trajectories of \eqref{syst schrod}, it is natural to consider them as prescribed. A solitary wave of \eqref{syst schrod} is a solution having the form
\[
\Psi_1(t,x) = e^{-i\la_1 t} u_1(x) \quad\text{and}\quad \Psi_2(t,x) = e^{-i\la_2 t} u_2(x)
\]
for some $\la_1,\la_2 \in \R$. This ansatz leads to \eqref{eq:1.0} for $(u_1,u_2)$ with $f_1(u_1)=g_1(|u_1|)u_1$, $f_2(u_2)=g_2(|u_2|)u_2$, and $F(u_1,u_2)=\frac12G(|u_1|^2,|u_2|^2)$.

The question of finding normalized solutions is already interesting for scalar equations and provides features and difficulties which are not present when the normalization condition is being dropped. Since the scalar setting will of course be relevant when treating systems, let us recall a few facts. Solutions $u\in H^1(\R^N)$ of
\beq[scalar]
-\De u = \la u + f(u),\quad \int_{\R^N} |u|^2 = a,
\eeq
with $a>0$ fixed can be obtained as critical points of the functional
\[
J(u) = \frac12 \int_{\R^N} |\nabla u|^2 - \int_{\R^N} F(u),
 \quad \text{with } \ F(s) = \int_0^s f(t)\, dt,
\]
constrained to the $L^2$-sphere $S_a:=\left\{u\in H^1(\R^N):\int_{\R^N} |u|^2 = a^2\right\}$, provided $f$ is subcritical. The model nonlinearity is $f(s)=|s|^{p-2}s$ with $2<p<2^*=\frac{2N}{N-2}$. The parameter $\la$ in the equation appears then as Lagrange multiplier.

The best studied cases of \eqref{eq:scalar} correspond to the situation when a solution can be found as a global minimizer of $J$ on $S(a)$ which is the case if $2<p<2+\frac4N$ for the model nonlinearity. This research mainly started with the work of Stuart \cite{St1,St2}. A bit later the Concentration Compactness Principle of P.L.~Lions \cite{Li1,Li2} was used in this type of problems. The case when $J$ is unbounded from below (and from above) on $S_a$, i.e.\ $2+\frac4N<p<2^*$ for the model nonlinearity, has already been much less studied. In this case a mountain pass structure has been exploited in \cite{Je} leading to the existence of one normalized solution. The existence of infinitely many normalized solution has later been proved in \cite{BaVa} where a "fountain" type structure on the $L^2$-sphere has been discovered which is somewhat reminiscent to the one for the free functional from \cite{Ba}; see also \cite{Wi}. More results on normalized solutions for scalar equations can be found in \cite{BeJe, BeJeLu, JeLuWa}. Technical difficulties in dealing with the constrained functional are that the existence of bounded Palais-Smale sequences requires new arguments, that Lagrange multipliers have to be controlled, and that weak limits of Palais-Smale sequences a-priori do not necessarily lie on the same $L^2$-sphere.

The goal of this paper is to find positive radial solutions
of systems like \eqref{eq:1.0} under various growth conditions on $f_1,f_2,F$. In order to keep the ideas and the results simple, and in order to avoid technicalities we only deal with homogeneous nonlinearities $f_1(s)=\mu_1|s|^{p_1-2}s$, $f_2(s)=\mu_2|s|^{p_2-2}s$, and $F(s,t)=\beta|s|^{r_1}|t|^{r_2}$. Thus we look for positive radial solutions
$u_1,u_2\in E:= H^1_{rad}(\R^N) \times H^1_{rad}(\R^N)$ of the system
\beq[1.1]
\left\{
  \begin{array}{ll}
   - \De u_1
    = \la_1 u_1 + \mu_1 |u_1|^{p_1 -2}u_1 + r_1 \be |u_1|^{r_1-2}|u_2|^{r_2}u_1 \smallskip\\
   - \De u_2
    = \la_2 u_2 + \mu_2 |u_2|^{p_2 -2}u_2 + r_2 \be |u_1|^{r_1}|u_2|^{r_2 -2}u_2
 \end{array}
\right.
\eeq
which are $L^2$-normalized in the sense of \eqref{eq:1.2}. Throughout the paper we require $N\geq2$, $p_1,p_2\in]2,2^*,[\,\setminus\left\{2+\frac4N\right\}$, and $\be,\mu_1,\mu_2,r_1,r_2,a_1,a_2>0$ with $2 \le r_1+r_2 < 2^*$. Thus we treat various self-focussing cases and attractive interaction. These constants are prescribed whereas the parameters $\la_1$ and $\la_2$ are unknown and will appear as Lagrange multipliers. The system comes from mean field models for binary mixtures of Bose-Einstein condensates or for binary gases of fermion atoms in degenerate quantum states (Bose-Fermi mixtures, Fermi-Fermi mixtures); see \cite{Ad,Bagnato-etal,esry-etal:1997,malomed:2008} and the references therein. The most famous case is the one of coupled Gross-Pitaevskii equations in dimension $N\le3$ where $p_1=p_2=4$, $r_1=r_2=2$ modeling Bose-Einstein condensation. However models for other ultracold quantum gases use different exponents.

The particular case $p_1=p_2=4$, $r_1=r_2=2$ of coupled Gross-Pitaevskii equations in $\R^3$ is being treated in the companion paper \cite{BaJeSo:2015}. In the present paper we deal with general exponents and distinguish between the cases $p_1,p_2<2+\frac4N$, $p_1<2+\frac4N<p_2$ and $p_1,p_2>2+\frac4N$. The exponent $2+\frac4N$ is critical for the normalized solution problem and will not be treated here. Other results on the existence of prescribed $L^2$-norm solutions for systems can be found in \cite{AlPa, AlBh, Ik, NoTaTeVe:2012, TaTe, NoTaVe}. In these papers the solutions obtained are global minimizers of the associated functional (e.g.\ in the defocusing repelling case $\mu_1,\mu_2,\be<0$), or only the case of small masses $a_1,a_2\sim0$ has been treated (as in \cite{NoTaVe}). In the latter paper the system included a trapping potential or was defined on a bounded domain. Requiring the masses to be small is a bifurcation type result.

Up to our knowledge the results of this paper and of its companion paper \cite{BaJeSo:2015} are the first results where one obtains normalized solutions for systems when the associated functional, here $J$, is unbounded from below on the constraint, and when the masses need not be small.

The paper is organized as follows: In Section~\ref{sec:results} we state and discuss our results. Section~\ref{sec:pre} contains some preliminary results, whereas Sections~\ref{sec:Th1} and \ref{sec:pf-main} are devoted to the proofs of the theorems from Section~\ref{sec:results}.

\section{Statement of Results}\label{sec:results}

We fix $N\geq2$, $p_1,p_2\in(2,2^*)$, and $\be,\mu_1,\mu_2,r_1,r_2,a_1,a_2>0$ with $2\le r_1+r_2<2^*$. We seek for solutions in the space $E:= H^1_{rad}(\R^N) \times H^1_{rad}(\R^N)$ of pairs of radial functions in $H^1(\R^N)$. Our first result on \eqref{eq:1.1}, \eqref{eq:1.2} deals with a case where it is possible to minimize the functional on the constraint.

\begin{thm}\label{th:min}
The problem \eqref{eq:1.1}, \eqref{eq:1.2} has, for some $\la_1,\la_2 < 0$, a solution $(u_1,u_2)\in E$ satisfying $u_1>0$, $u_2>0$ in each of the following cases:
\begin{itemize}
\item[a)] $2\le N \le 4$ and $p_1,p_2,r_1+r_2 < 2+\frac4N$.
\item[b)] $N\ge5$ and $p_1,p_2<2+\frac{2}{N-2}$ and $r_1+r_2 < 2+\frac4N$.
\end{itemize}
\end{thm}

We do not know whether Theorem~\ref{th:min}~a) holds true for all $N\ge2$, i.e.\ whether the hypothesis $p_1,p_2<2+\frac{2}{N-2}$ in \ref{th:min}~b) can be replaced by $p_1,p_2<2+\frac4N$. Setting $S(a) = S_a\cap H^1_{rad}(\R^N) = \{u\in H^1_{rad}(\R^N):|u|_2^2 = a\}$, the solution in Theorem~\ref{th:min} will be a minimizer of the functional
\[
J(u_1,u_2)
 = \frac12\int_{\R^N} |\nabla u_1|^2 + |\nabla u_2|^2 \, dx
     - \int_{\R^N} \frac{\mu_1}{p_1}|u_1|^{p_1}
     + \frac{\mu_2}{p_2}  |u_2|^{p_2}
     + \be |u_1|^{r_1}|u_2|^{r_2} \, dx
\]
constrained to $S(a_1)\times S(a_2) \subset E$.

It is easy to prove that any minimizing sequence $\{(u_1^n,u_2^n)\} \subset S(a_1) \times S(a_2)$ associated to $J$ is bounded. Thus we can assume without restriction that $(u_1^n,u_2^n)\weakto(u_1,u_2)$ weakly in $E$ for some $(u_1,u_2) \in E$. From the weak convergence in $E$ we deduce that $(u_1,u_2)$ satisfies \eqref{eq:1.1} for some associated $(\la_1,\la_2)$. To prove Theorem~\ref{th:min} one still has to show that $(u_1,u_2) \in S(a_1)\times S(a_2)$. Even if we work in the space of radially symmetric functions this question is, with respect to the scalar case, challenging as was already observed in \cite{Ik}. Our proof of Theorem~\ref{th:min} ultimately relies on the use of a Liouville's type result for an associated scalar equation. This is responsible for the restriction that $N \leq 4$ in part~a), or that $p_1,p_2<2+\frac{2}{N-2}$ in part~b).

Our second result deals with the case where $p_2$ and $r_1 + r_2$ are bigger than $2+\frac4N$ so that $J$ is unbounded from below and minimization does not work. We require the following hypotheses on the coefficients.

\begin{itemize}
\item[(H1)]  $2 < p_1 < 2 + \frac4N < p_2 < 2^* $.
\item[(H2)] $2+\frac4N < r_1 + r_2 < 2^*$,\ \ $r_2>2$.
\end{itemize}

Consider the functional $I:H_{rad}^1(\R^N)\to\R$ defined by
$$
I(u) = \frac12\int_{\R^N} |\nabla u|^2 \,dx - \frac{\mu}{p} \int_{\R^N} |u|^p \,dx
$$
constrained to the $L^2$-sphere $S(a)$. For $p\in]2,2^*[\,\setminus\{2+\frac4N\}$ we denote by $m_p^{\mu}(a)$ the ground state level, i.~e.\
$$
m_p^{\mu}(a) = \inf \{ I(u): u \in S(a) \mbox{ such that } (I_{|S(a)})'(u)=0 \}.
$$
We discuss the properties of this ground state level in Lemma~\ref{lem-ground} below.

\begin{thm}\label{th:main}
Assume (H1), (H2) and $2 \le N \le 4$. If
\begin{equation}\label{level}
m_{p_1}^{\mu_1}(a_1) + m_{p_2}^{\mu_2}(a_2) <0,
\end{equation}
then, for some $\la_1 < 0$ and  $\la_2 < 0$, there exists a solution $(u_1, u_2) \in E$ of \eqref{eq:1.1}, \eqref{eq:1.2}, satisfying $u_1>0$, $u_2>0$.
\end{thm}

As a corollary of Theorem \ref{th:main} we obtain :

\begin{cor}\label{cor:main}
Assume (H1), (H2) and $2 \le N \le 4$. Then for any $a_2 >0$ there exists $\bar a_1 >0$ such that for any $a_1 \geq \bar a_1$ there exists a positive solution $(u_1, u_2) \in E$ of \eqref{eq:1.1}, \eqref{eq:1.2}, for some $\la_1 < 0$ and  $\la_2 < 0$. In addition $\bar a_1 \to 0$ as $a_2 \to \infty$.
\end{cor}

With respect to Theorem \ref{th:min} the proof of Theorem \ref{th:main} presents new difficulties. First one needs to identify a possible critical level $\gamma(a_1,a_2)$ where one can find Palais-Smale sequences. The construction of this minimax level, which is of mountain pass type, is the heart of the proof and is carried out in Lemmas \ref{lem:def-c}, \ref{lem:end-point} and \ref{intersection}. By Ekeland's variational principle there exists a Palais-Smale sequence associated to $\gamma(a_1,a_2)$. One then needs to find a bounded Palais-Smale sequence. We manage to find a special Palais-Smale sequence $\{(u_1^n,u_2^n)\} \subset S(a_1) \times S(a_2)$ having the additional property that $Q(u_1^n, u_2^n) \to 0$ where $Q: E \to \R$ is given by
\beq[def-Q]
\begin{aligned}
Q(u_1,u_2)
 &= |\nabla u_1|_2^2 + |\nabla u_2|_2^2
    - \frac{\mu_1}{p_1} N\left( \frac{p_1}{2}-1\right) |u_1|_{p_1}^{p_1}\\
 &\hspace{1cm}
    - \frac{\mu_2}{p_2} N\left(\frac{p_2}{2}-1\right) |u_2|_{p_2}^{p_2}
    - N \be\left(\frac{r_1+r_2}{2}-1\right) \int_{\R^N} |u_1|^{r_1}|u_2|^{r_2} \,dx.
\end{aligned}
\eeq
The condition $Q(u_1,u_2)=0$ corresponds to a natural constraint of Pohozaev type on the solutions of \eqref{eq:1.1}, \eqref{eq:1.2}; see Remark \ref{natural}. To construct a Palais-Smale sequence having the additional property $Q(u_1^n,u_2^n) \to 0$ we employ similar arguments as in \cite{BaVa,BeJeLu,Je,Lu}; see also \cite{AzAvPo,HiIkTa}.

From the property that $Q(u_1^n,u_2^n) \to 0$ we deduce that $\{( u_1^n,u_2^n)\} \subset E$ is bounded. Finally in order to insure the strong convergence of our Palais-Smale sequence we combine the estimate \eqref{level} with the Liouville argument already used in the proof of Theorem 1.1.

In our last result we assume the inequalities $p_1,p_2, r_1 + r_2>2 + \frac4N$.

\begin{thm}\label{th:max}
Assume that $p_1,p_2, r_1 + r_2 > 2 + \frac4N$ and that $2 \leq N \leq 4$.
\begin{itemize}
\item[a)] There exists $\be_1=\be_1(a_1,a_2,\mu_1,\mu_2)>0$ such that if $\be\le\be_1$ then \eqref{eq:1.1}, \eqref{eq:1.2} has a positive solution $(u_1,u_2) \in E$ for some $\la_1<0$ and $\la_2<0$.
\item[b)] There exists $\be_2=\be_2(a_1,a_2,\mu_1,\mu_2)>0$ such that if $\be\ge\be_2$ then \eqref{eq:1.1}, \eqref{eq:1.2} has a positive solution $(u_1,u_2) \in E$ for some $\la_1<0$ and $\la_2<0$.
\end{itemize}
\end{thm}

We would like to emphasize that the proof yields explicit estimates for $\be_1$ from below and $\be_2$ from above in terms of $p_1,p_2,r_1,r_2$ and $a_1,a_2,\mu_1,\mu_2$, in particular $\be_1$ and $\be_2$ are not obtained by limiting processes.

Theorem~\ref{th:max} is a generalization of the result from \cite{BaJeSo:2015} where the case $N=3$, $p_1=p_2=4$, $r_1=r_2=2$ has been considered. The proof of Theorem~\ref{th:max}~a) is based on a two-dimensional linking on the constraint $S=S(a_1)\times S(a_2)$ whereas the proof of Theorem~\ref{th:max}~b) uses a mountain pass argument on $S$. As in Theorem \ref{th:main} one obtains a special Palais-Smale sequence $\{(u_1^n,u_2^n)\} \subset S(a_1) \times S(a_2)$ at the suspected critical level $\theta(a_1,a_2)$ such that $Q(u_1^n,u_2^n) \to 0$. This leads in particular to its boundedness. In order to obtain the strong convergence an upper bound for $\beta$ is needed in part a), and a lower bound in part b). Concerning estimates for $\be_1$ or $\be_2$ we just mention that in the setting of \cite{BaJeSo:2015} one has  $\be_1\to\infty$ if $\mu_1=\mu_2\to\infty$ and $a_1,a_2$ being fixed. Similarly, $\be_2\to0$ if $\mu_1=\mu_2\to0$ and $a_1,a_2$ are fixed.  Since the proof in \cite{BaJeSo:2015} for the special case $N=3$, $p_1=p_2=4$, $r_1=r_2=2$, generalizes easily we simply refer to \cite{BaJeSo:2015} and do not give any details here.

\begin{remark} The results presented in this paper for $N \ge 2$ can be extended to $N=1$. The difference between the cases $N=1$ and $N \geq 2$ is that the compact embedding $H_{rad}^1(\R^N) \subset L^q(\R^N)$ for $q \in ]2, 2^*[$ only holds when $N \geq 2$. When $N=1$  the corresponding compactness can however be regained by working with Palais-Smale sequences of almost Schwartz-symmetric functions.
In order to avoid additional technicalities we do not deal with the case $N=1$ in this paper but instead refer the reader to \cite{JeLuWa} where a similar issue is treated. The results are identical in the cases $N=1$ and $N \geq 2$ except that in the case $N=1$ one should require in addition that $r_2 >4$ in (H2) (this restriction originates only from the adapted version of Lemma~\ref{lem:def-c}).
\end{remark}

\section{Preliminary results}\label{sec:pre}

Throughout the paper we denote by $H$ the space $H^1_{rad}(\R^N)$ equipped with the standard norm $|\cdot|$, so $E=H\times H$. We also denote by $\cS$ the constraint $S(a_1) \times S(a_2)$. We recall, see for example \cite{BeLi}, that if $u_n \rightharpoonup u$ weakly in $H$ then $u_n \to u$ strongly in $L^q(\R^N)$ for $q \in ]2,2^*[$. \medskip

Let us first observe that the functional $J$ is well defined. For $2 \le r_1+r_2 \le 2^*$ there exists $q>1$ with
\beq[def-q]
\max\left\{\frac{2}{r_1},\frac{2^*}{2^*-r_2}\right\} \le q \le
 \min\left\{\frac{2^*}{r_1},\frac{2}{(2-r_2)^+}\right\},
\eeq
which implies $2 \le r_1q,r_2q' \le 2^*$, hence
\[
\int_{\R^N} |u_1|^{r_1}|u_2|^{r_2} \, dx
 \le |u_1|_{r_1q}^{r_1}\cdot|u_2|_{r_2q'}^{r_2}
 < \infty.
\]
The Gagliardo-Nirenberg inequality
\[
|u|_p \le C(N,p)|\nabla u|_2^\al \cdot|u|_2^{1-\al}\quad
 \text{where } \al=\frac{N(p-2)}{2p}
\]
which holds for $u \in H^1(\R^N)$ and $2 \le p \le 2^*$, implies for $u_1\in S(a_1)$, $u_2\in S(a_2)$:
\beq[est0]
\int_{\R^N} |u_1|^{p_1} \le C(N,p_1,a_1)|\nabla u_1|_2^{\frac{N(p_1-2)}{2}},\qquad
\int_{\R^N} |u_2|^{p_2} \le C(N,p_2,a_2)|\nabla u_2|_2^{\frac{N(p_2-2)}{2}},
\eeq
and
\beq[mixed-est1]
\int_{\R^N} |u_1|^{r_1}|u_2|^{r_2} \, dx
 \le |u_1|_{r_1q}^{r_1}\cdot|u_2|_{r_2q'}^{r_2}
 \le C|\nabla u_1|_2^{\frac{N(r_1q-2)}{2q}}|\nabla u_2|_2^{\frac{N(r_2q'-2)}{2q'}}
\eeq
with $C=C(N,r_1,r_2,a_1,a_2,q)$.

\begin{lem}\label{lem-ground}
Assume that $p \in ]2, 2^*[ \setminus \left\{2+\frac4N\right\}$, and let $\mu>0$ be given. For any $a>0$ there exists a unique couple $(\la_a,u_a) \in \R^+ \times H$ solving
\beq[ground]
- \De u + \la u = \mu |u|^{p-2}u, \quad u \in H^1(\R^N),
\eeq
and such that $u_a > 0$ and $|u_a|_2^2 = a$. Moreover $u_a$ corresponds to the least energy level $m_p^{\mu}(a)$ of the functional
$I : H \to \R$ defined by
$$
I(u) = \frac12\int_{\R^N} |\nabla u|^2 \,dx - \frac{\mu}{p} \int_{\R^N} |u|^p \,dx
$$
constrained to the $L^2$-sphere $S(a)$. If $p \in ]2, 2 + \frac{4}{N}[$ then $m_p^{\mu}(a) <0$ for all $a>0$, the map $a \mapsto m_p^{\mu}(a)$ is strictly decreasing, and $m_p^{\mu}(a) \to - \infty$ as $a \to \infty$.  If $p \in ] 2 + \frac{4}{N}, 2^*[$ then $m_p^{\mu}(a) >0$ for all $a >0$, the map $a \mapsto m_p^{\mu}(a)$ is strictly decreasing and $m_p^{\mu}(a) \to 0$ as $a \to \infty$.
\end{lem}

\begin{proof}
It is standard (see \cite{Kw}) that the equation
\begin{equation}\label{2.5}
- \De u + \la u = \mu |u|^{p-2}u, \quad u \in H^1(\R^N),
\end{equation}
has, for any $\la >0$, a unique positive radial solution $u_{\la}$. By direct calculations one can show that $u_{\la}$ is given by
$$ u_{\la}(x) = \la^{\frac{1}{p-2}}w(\sqrt{\la}x)$$
where $w$ is the unique positive radial solution of
$$ - \De u + u = \mu |u|^{p-2}u, \quad u \in H^1(\R^N).$$
Since
$$|u_{\la}|_2^2 = \la^{\big(\frac{2}{p-2}- \frac{N}{2}\big)} |w|_2^2$$
for any $a>0$ there exists a unique $\la_a >0$, explicitely given by
$$\la_a = \Big(\frac{a}{|w|_2^2}\Big)^{\frac{2(p-2)}{4 - N(p-2)}},$$
such that $u_{\la_a} \in H^1(\R^N)$ satisfies $|u_{\la_a}|_2^2 = a$ and is the unique positive solution of
$$- \De u + \la_a u = \mu |u|^{p-2}u, \quad u \in H^1(\R^N).$$
The solution $u_{\la_a}$ corresponds to a least energy solution of the functional
$I : H \to \R$ defined by
$$
I(u) = \frac12\int_{\R^N} |\nabla u|^2 \,dx - \frac{\mu}{p} \int_{\R^N} |u|^p \,dx
$$
constrained to the $L^2$-sphere $S(a)$. Here $\la_a >0$ appears as the associated Lagrange parameter. To prove this statement two cases have to be distinguished: \medskip

\noindent{\it Case 1 : $ p \in ]2, 2 + \frac{4}{N}[$.} \\
The least energy solution corresponds to the energy level
$$m_p^{\mu}(a) = \inf_{u \in S(a)} I(u).$$
It is standard \cite{St1,St2}, see also \cite{BeBeGhMi}, that $m_p^{\mu}(a) <0$, that the map $a \mapsto m_p^{\mu}(a)$ is strictly decreasing, and that $m_p^{\mu}(a) \to - \infty$ as $a \to \infty$. \medskip

\noindent{\it Case 2 : $ p \in ]2 + \frac{4}{N}, 2^*[$.} \\
The least energy solution corresponds to the energy level
$$m_p^{\mu}(a) = \inf_{u \in V(a)} I(u).$$
Here
\begin{equation}\label{eq:def-V}
V(a) = \Big\{ u \in S(a) : |\nabla u|_2^2 = \frac{N(p-2)}{2}|u|_p^p \Big\}
\end{equation}
is a natural constraint which contains all the critical points of $I$ constrained to $S(a)$. This has been proved in \cite[Lemma 9.3]{BeJeLu}, see also \cite{Je}. Also in \cite[Lemma 9.3]{BeJeLu} it is shown that the map $a \mapsto m_p^{\mu}(a)$ is strictly decreasing and that $m_p^{\mu}(a) \to 0$ as $a \to \infty$.
\end{proof}

\begin{lem}\label{Palais-Smale}
Let $\{(u_1^n, u_2^n)\} \subset \cS$ be a bounded Palais-Smale sequence of $J$ restricted to $\cS$. Then there exist $(u_1,u_2) \in E$, $(\la_1, \la_2) \in \R \times \R$ and a sequence $\{(\la_1^n, \la_2^n)\} \in \R \times \R$ such that, up to a subsequence:
\begin{itemize}
\item[a)] For $i=1,2$, $u_i^n \weakto u_i$ weakly in $H$ and in $L^2(\R^N)$,\,  $u_i^n \to u_i$ in $L^q(\R^N)$ for any $q \in ]2, 2^*[$;
\item[b)] $(\la_1^n, \la_2^n)  \to (\la_1,\la_2)$ in  $\R \times \R$;
\item[c)] $J'(u_1^n, u_2^n) - \la_1^n (u_1^n,0) - \la_2^n (0, u_2^n) \to 0$ in $E^*$;
\item[d)] $(u_1,u_2)$ is solution of the system \eqref{eq:1.1} where $(\la_1,\la_2)$ are given in b).
\end{itemize}
In addition if $\la_1<0$ then $u_1^n \to u_1$ strongly in $H$. Similarly if $\la_2<0$ then
$u_2^n\to u_2$ strongly in $H$.
\end{lem}

\begin{proof}
Point a) is trivial. Since $\{(u_1^n, u_2^n)\} \subset H \times H$ is bounded, following Berestycki and Lions \cite[Lemma 3]{BeLi}, we know that $(J|_{\cS})'(u_1^n, u_2^n) \to 0$ in $E^*$ is equivalent to
\[
 J'(u_1^n, u_2^n)- \frac{1}{|u_1^n|_2^2}\langle J'(u_1^n,u_2^n),(u_1^n,0) \rangle (u_1^n,0)
 - \frac{1}{|u_2^n|_2^2}\langle J'(u_1^n,u_2^n),(0,u_2^n) \rangle (0,u_2^n)
 \to 0
\]
in $E^*$. Therefore we obtain
$$
 J'(u_1^n, u_2^n)- \la_1^n (u_1^n,0) - \la_2^n (0,u_2^n) \longrightarrow 0 \quad \text{in } E^*
$$
with
\begin{equation}
\la_1^n= \frac{1}{|u_1^n|_2^2}\left( |\nabla u_1^n| - \mu_1 |u_1^n|_{p_1}^{p_2} - \beta r_1 \int_{\R^N}|u_1^n|^{r_1}|u_2^n|^{r_2} dx\right)
\end{equation}
and
\begin{equation}
\la_2^n= \frac{1}{|u_2^n|_2^2}\left( |\nabla u_2^n| - \mu_1 |u_2^n|_{p_1}^{p_2} - \beta r_2 \int_{\R^N}|u_2^n|^{r_2}|u_1^n|^{r_1} dx\right).
\end{equation}
This proves point c). To prove point b), namely that $\{(\la_1^n,\la_2^n)\} \subset \R\times\R$ is bounded, it suffices to recall that $\{(u_1^n,u_2^n)\} \subset E$ is bounded and to use the estimates \eqref{eq:est0} and \eqref{eq:mixed-est1}. Now from points b) and c) it is standard to deduce d).

It remains to show that if $\la_1 <0$ then $u_1^n \to u_1$ strongly in $H^1(\R^N)$, and in particular in $L^2(\R^N)$. Since
$$
|u_1^n|_{p_1}^{p_1} \to |u_1|_{p_1}^{p_1} \quad \mbox{ and  } \quad \int_{\R^N}|u_1^n|^{r_1}|u_2^n|^{r_2} \, dx \to \int_{\R^N}|u_1|^{r_1}|u_2|^{r_2} \, dx,
$$
and using the fact that $\langle J'(u_1^n, u_2^n) - \la_1^n(u_1^n,0),(u_1^n,0)\rangle  \to \langle J'(u_1,u_2) - \la_1 (u_1,0), (u_1,0)\rangle =0 $, we deduce
\beq[lim]
|\nabla u_1^n|_2^2   - \la_1^n |u_1^n|_2^2  \to |\nabla u_1|_2^2 - \la_1 |u_1|_2^2.
\eeq
As a consequence of the weak convergence $u_i^n \weakto u_i$ we obtain
$$|\nabla u_1|_2^2 \leq \liminf |\nabla u_1^n|_2^2 \quad \mbox{and} \quad |u_1|_2^2  \leq \liminf |u_1^n|_2^2.$$
Finally, since $\la_1^n \to \la_1$ we deduce from \eqref{eq:lim} that
$$|\nabla u_1^n|_2^2 \to |\nabla u_1|_2^2 \quad \mbox{and} \quad | u_1^n|_2^2 \to | u_1|_2^2$$
from which the strong convergence follows.  The case of $\la_2 <0$ is treated in the same way.
\end{proof}

\begin{lem} \label{lemIkoma}
\begin{itemize}
\item[a)] Suppose that $q\in ]1,\frac{N}{N-2}]$ when $N \geq 3$ and $q\in]1,\infty[$ when $N=1,2$. Let $u\in L^q(\R^N)$ be a smooth nonnegative function satisfying $-\De u \ge 0$ in $\R^N$. Then $u \equiv 0$ holds.
\item[b)] For $1<q\le1+\frac2{N-2}$ the inequality $-\De u \ge u^q$ does not have a positive classical solution in $\R^N$.
\end{itemize}
\end{lem}

\begin{proof}
a) can be found in \cite[Lemma A.2]{Ik}; b) is due to \cite{Gi}, a simple proof can be found in \cite{QuSo}.
\end{proof}

\begin{lem}\label{Step5}
Assume $N \leq 4$, or $N\ge5$ and $p_1\le2+\frac2{N-2}$. If $(u_1,u_2)\in E$ is a solution of \eqref{eq:1.1} with $u_1 \gneqq 0$ and $u_2 \ge 0$, then $\la_1 <0$. If $(u_1,u_2)\in E$ is a solution of \eqref{eq:1.1} with $u_2 \gneqq 0$ and $u_1 \ge 0$, then $\la_2 <0$.
\end{lem}

\begin{proof}
In the first case since $u_1 \gneqq 0$ satisfies
$$
- \De u_1 =  \la_1 u_1 + \mu_1 u_1^{p_1-1} + r_1\be u_1^{r_1-1} u_2
$$
and since all summands on the right hand side are non negative if $\la_1 \geq 0$, we conclude by Lemma~\ref{lemIkoma} that $u_1=0$. This contradicts the assumption that $u_1 \gneqq 0$. The proof of the other part is identical.
\end{proof}

\section{Proof of Theorem~\ref{th:min}}\label{sec:Th1}

\begin{lem}\label{PSth1}
If $p_1,p_2,r_1+r_2<2+\frac4N$ then $J$ is bounded from below and coercive on $\cS$ for any $a_1,a_2>0$. In addition there exists a bounded Palais-Smale sequence $\{(u_1^n,u_2^n)\}\subset\cS$ which satisfies $(u_1^n)^- \to 0$ and $(u_2^n)^- \to 0$  in $H$. Here $(u_i^n)^- = max\{0,-u_i^n\}$ for $i=1,2.$
\end{lem}

\begin{proof}
Observe that $\frac{N(p_i-2)}{2} < 2$ because $p_i < 2+\frac4N$, $i=1,2$, and that
\beq[mixed-est2]
\frac{N(r_1q-2)}{2q} + \frac{N(r_2q'-2)}{2q'} < 2
\eeq
since $r_1+r_2<2+\frac4N$. It follows easily  from \eqref{eq:est0}, \eqref{eq:mixed-est1} and \eqref{eq:mixed-est2} that $J$ is bounded below and coercive on $\cS$. \medskip

Now let $\{(v_1^n, v_2^n)\} \subset \cS$ be a minimizing sequence for $J$ on $\cS$. By the coerciveness of $J$ it is bounded and also without restriction we can assume that $v_1^n \geq 0$ and $v_2^n \geq 0$. Using Ekeland's variational principle \cite{Ek,Gh} we deduce that there exists a minimizing sequence $\{(u_1^n, u_2^n)\} \subset S$ which is a Palais-Smale sequence for $J$ restricted to $\cS$ and which satisfies $(u_1^n,u_2^n) - (v_1^n,v_2^n) \to 0$ in $E$. In particular $(u_1^n)^- \to 0$ and $(u_2^n)^- \to 0$ in $H$.
\end{proof}

\begin{proof}[Proof of Theorem \ref{th:min}]
From Lemmas \ref{lem-ground}  and \ref{PSth1} we deduce the existence of a bounded Palais-Smale sequence $\{(u_1^n, u_2^n)\} \subset \cS$ such that $(u_1^n,u_2^n) \rightharpoonup (u_1,u_2)$ weakly in $E$ with $u_1 \geq 0$ and $u_2 \geq 0$. We also obtain a couple $(\la_1,\la_2) \in \R \times \R$ for which $(u_1,u_2)$ is solution of the system \eqref{eq:1.1}. To conclude the proof it remains to show that $u_1^n \to u_1$ and $u_2^n \to u_2$ in $H$. Indeed if this is the case then we both have $u_1 \in S(a_1)$ and $u_2 \in S(a_2)$ and that $(u_1,u_2)$ is a least energy solution. In addition by the strong maximum principle, applied separately to each equation, we obtain that $u_1 >0$ and $u_2 >0$. In order to show the strong convergence in $H$ we define
$$m(a_1,a_2): = \inf_{(u_1,u_2) \in \cS} J(u_1,u_2).$$
Since $\beta \geq 0$ we clearly have
\beq[estminimum]
m(a_1,a_2) \leq m_{p_1}^{\mu_1}(a_1) + m_{p_2}^{\mu_2}(a_2) <0
\eeq
where the last inequality comes from Lemma \ref{lem-ground}. We now distinguish four cases and we show that only the last one may occur: \medskip

\noindent{\it Case 1: $u_1 = 0$ and $u_2 = 0$.} \\
Then $|u_1^n|_{p_1}^{p_1} \to 0$, $|u_2^n|_{p_2}^{p_2} \to 0$ and $\int_{\R^N}|u_1^n|^{r_1}|u_2^n|^{r_2}\, dx \to 0$. Thus
$\limsup J(u_1^n,u_2^n) \geq 0$ which contradicts \eqref{eq:estminimum}. \medskip

\noindent{\it Case 2: $u_1 = 0$ and $u_2 \neq 0$.} \\
Then $$\limsup J(u_1^n,u_2^n) \geq \frac{1}{2}|\nabla u_2|_2^2 - \frac{\mu_2}{p_2}|u_2|_{p_2}^{p_2} \geq m(\bar{a}_2)$$
where $\bar{a}_2 := |u_2|_2^2 \le a_2$. By Lemma~\ref{lem-ground} we know that
$m(\bar{a}_2) \ge m(a_2)$, and since $m(a_1)<0$ we have a contradiction with \eqref{eq:estminimum}. \medskip

\noindent{\it Case 3: $u_1 \ne 0$ and $u_2 = 0$.} \\
Reversing the r\^ole of $u_1$ and $u_2$ we obtain a contradiction similar to case 2. \medskip

\noindent{\it Case 4: $0<|u_1|_2^2=\bar{a}_1\le a_1\,$ or $\,0<|u_2|_2^2=\bar{a}_2\le a_2$.}\\
Necessarily this case occurs. Now using Lemma \ref{Step5} we deduce that $\la_1<0$ and $\la_2<0$. Then Lemma~\ref{Palais-Smale} implies $u_1^n \to u_1$ and $u_2^n \to u_2$ in $H$. At this point the proof of the theorem is completed.
\end{proof}

\section{Proof of Theorem~\ref{th:main}}\label{sec:pf-main}

For $c >0$ we define the sets
$$ A_c = \{ u_2 \in S(a_2) : |\nabla u_2|_2^2 \leq c \} \quad \mbox{and}
 \quad B_c = \{ u_2 \in S(a_2) : |\nabla u_2|_2^2 = 2 c \}.$$
Setting $J_{u_1}(u_2) := J(u_1,u_2)$ for $u_1 \in S(a_1)$ we observe that
$$
J_{u_1} (u_2)
 = J_{u_1}(0) + \frac12 \int_{\R^N} |\nabla u_2|^2 \, dx
   - \frac{\mu_2}{p_2} \int_{\R^N} |u_2|^{p_2}\,dx - \be \int_{\R^N} |u_1|^{r_1} |u_2|^{r_2}\, dx.
$$

\begin{lem}\label{lem:def-c}
There exists a continuous function $c : S(a_1) \to \R$, $u_1 \to c(u_1)$, such that
\[
\sup_{A_{c(u_1)}} J_{u_1} < \inf_{B_{c(u_1)}} J_{u_1} \qquad
 \text{for all $u_1 \in S(a_1)$.}
\]
The function $c$ is bounded, and it is bounded away from $0$ on bounded subsets of $S(a_1)$.
\end{lem}

\begin{proof}
Fixing $u_1 \in S(a_1)$ we first observe that for $u_2 \in A_c$ there holds:
$$
J_{u_1}(u_2) \leq J_{u_1}(0) + \frac12\int_{\R^N} |\nabla u_2|^2 \, dx
 \leq J_{u_1}(0) + \frac12c.
$$
For $u_2 \in B_c$ and $q$ as in \eqref{eq:def-q}, and $\ga=\frac{N(r_2q'-2)}{2q'}$
we have, using the Gagliardo-Nirenberg inequality, see \eqref{eq:est0}, \eqref{eq:mixed-est1},
$$
\begin{aligned}
J_{u_1}(u_2)
 &\ge
   J_{u_1}(0) + c
    - \frac{\mu_2}{p_2} c(p_2,N)|\nabla u_2|_2^{\frac{N}{2}(p_2-2)}
       |u_2|_2^{p_2(1-\frac{N}{2})+N}
    - \be |u_1|_{r_1q}^{r_1}\cdot|u_2|_{r_2q'}^{r_2} \\
 &= J_{u_1}(0) + c - K_1c^{\frac{N}{4}(p_2-2)}
      - K_2|u_1|_{r_1q}^{r_1}\cdot c^{\frac{\gamma}2}
\end{aligned}
$$
Here $K_1=K_1(N,\mu_2,p_2,a_2)$ and $K_2=K_2(N,\be,r_2,a_2,q)$. Observe that $\frac{N}{4}(p_2-2)>1$ because $p_2>2+\frac4N$, and $\ga>2$ provided $q<\frac{2N}{2N-r_2N+4}$. We can choose $q$ satisfying this inequality and \eqref{eq:def-q} because
\[
\frac{2N}{2N-r_2N+4} > \max\left\{\frac{2}{r_1},\frac{2^*}{2^*-r_2}\right\}
\]
which is a consequence of $r_1+r_2>2+\frac4N$ and  $r_2>2$.

Observe that $K_1c^{\frac{N}{4}(p_2-2)} \le \frac18c$ if $c>0$ is small because $\frac{N(p_2-2)}{4}>1$, and $K_2|u_1|_{r_1q}^{r_1}\cdot c^{\frac\ga2} \le \frac18c$ if $c>0$ is small because $\ga>2$. More precisely, if $c:S(a_1)\to\R^+$ satisfies
\beq[def-c]
c(u_1) \le \min\left\{(8K_1)^{-\frac{4}{N(p_2-2)-4}},
             (8K_2)^{-\frac{2}{\ga-2}}\cdot|u_1|_{r_1q}^{-\frac{2r_1}{\ga-2}}\right\},
\eeq
then we have for $u_2\in B_{c(u_1)}$:
\beq[JonB-est]
\begin{aligned}
J_{u_1}(u_2)
 &\ge  J_{u_1}(0) + c(u_1) - \frac18c(u_1) - \frac18c(u_1)\\
 &> J_{u_1}(0) + \frac12c(u_1)
 \ge \sup_{A_{c(u_1)}} J_{u_1}.
\end{aligned}
\eeq
Clearly we may define a continuous function $c:S(a_1)\to\R^+$ satisfying \eqref{eq:def-c} and which is bounded away from $0$ on bounded subsets of $S(a_1)$. In fact, the right hand side of \eqref{eq:def-c} may serve as definition. By \eqref{eq:def-c} $c$ is also bounded above.
\end{proof}

Now we set
$$A(u_1) = A_{c(u_1)}, \quad B(u_1) = B_{c(u_1)}$$
and
$$B = \{ (u_1,u_2): u_1 \in S(a_1),\, u_2 \in B(u_1)\}.$$
Let $\underline{u} \in S(a_1)$ be such that
\begin{equation}\label{star5}
J(\underline{u},0) = \min_{u \in S(a_1)} J(u,0) <0.
\end{equation}
The existence of $\underline{u}$ is insured by Lemma \ref{lem-ground}.

\begin{lem}\label{lem:end-point}
There exist $\overline{v} \in A(\underline{u})$ and $\overline{w} \in S(a_2)\setminus A_{2c(\underline{u})}$ such that
$$
\max \{J(\underline{u}, \overline{v}), J(\underline{u}, \overline{w}) \}
 < \inf_{(u_1,u_2) \in B}  J(u_1,u_2).
$$
\end{lem}

\begin{proof}
Since $J(\underline{u},u_2) \to J(\underline{u},0)$ as $|\nabla u_2|_2 \to 0$, in order to obtain
$\overline{v} \in A(\underline{u})$ it is sufficient to prove $J(\underline{u},0)<\inf_BJ$. The functional
$J(\cdot,0): S(a_1)\to\R$ is coercive because $2<p_1<2+\frac4N$. Choose $R>0$ such that $J(u_1,0)\ge J(\underline{u},0)+1$ if $|\nabla u_1|_2\ge R$. Then we have for $(u_1,u_2)\in B$ with
$|\nabla u_1|_2\ge R$, cf.~\eqref{eq:JonB-est}:
\[
J(u_1,u_2) \ge J(u_1,0)+\frac34c(u_1) > J(\underline{u},0)+1\,.
\]
By Lemma~\ref{lem:def-c} there holds
\[
\eps:=\inf_{|\nabla u_1|_2\le R} c(u_1) > 0
\]
which implies for $(u_1,u_2)\in B$ with $|\nabla u_1|_2\le R$:
\[
J(u_1,u_2) \ge J(u_1,0)+\frac34c(u_1) \ge J(\underline{u},0)+\frac34\eps\,.
\]

In order to find $\overline{w} \in S(a_2)\setminus A_{2c(\underline{u})}$ as required we define for each $u\in S(a_2)$ and $t\in\R$ the scaled function $t*u$ by $(t*u)(x) = e^{t \frac{N}{2}}u(e^tx)$.
Clearly $t*u\in S(a_2)$ for every $t>0$, and
$|\nabla(t*u)|_2\to\infty$ as $t\to\infty$. Now since $p_2 > 2 + \frac4N$,
fixing an arbitrary $u \in S(a_2)$ we see that $J(\underline{u}, (t*u)) \to - \infty$
as $t \to \infty$.
\end{proof}

As a consequence of Lemma~\ref{lem:end-point} the set
$$
\begin{aligned}
\Ga
 &:= \big\{g\in\cC([0,1],\cS):g(0)=(v_1,v_2),\, g(1) = (w_1, w_2),\\
 &\hspace{2cm}
       v_2\in A(v_1),\,w_2\notin A_{2c(w_1)},\,\max\{J(v_1,v_2),J(w_1,w_2)\}<\inf_BJ\big\}
\end{aligned}
$$
is nonempty.

\begin{lem}\label{intersection}
We have
$$
\ga(a_1, a_2) := \inf_{g\in\Ga} \max_{t\in[0,1]} J(g(t)) \geq \inf_BJ.
$$
\end{lem}

\begin{proof}
We just need to show that for each $g(t) = (g_1(t),g_2(t)) \in \Ga$ there exists a
$t \in [0,1]$ such that $g(t) \in B$. The map $\al : [0,1] \to \R$ given by
$t \to |\nabla g_2(t)|_2^2 - 2 c(g_1(t))$ satisfies
$$
\al (0) = || \nabla v_2||_2^2 - 2 c(v_1) \leq c(v_1) - 2 c(v_1) < 0
$$
and
$$
\al (1) = ||\nabla w_2||_2^2 - 2 c(w_1) >0.
$$
Thus there exists a $t \in [0,1]$ such that $\al (t) = 0$, which means $g(t) \in B$.
\end{proof}

For future reference we also need.

\begin{lem}\label{upperbound}
Assume that (H1) and (H2) hold. Then for any $a_1>0$ and $a_2>0$ we have
\begin{equation}\label{upperestimate}
\gamma(a_1,a_2) \leq m_{p_1}^{\mu_1}(a_1) + m_{p_2}^{\mu_2}(a_2).
\end{equation}
\end{lem}

\begin{proof}
Let $\overline{u} \in S(a_2)$ be such that
$$J(0, \overline{u}) = I(\overline{u}) = \min_{u \in V(a_2)}I(u) = m_{p_2}^{\mu_2}(a_2)$$
whose existence and characterization is recalled in Lemma~\ref{lem-ground}, with $V(a)$ defined in \eqref{eq:def-V}. Since $\overline{u} \in V(a_2)$ it is readily seen that
\begin{equation}\label{star6}
\max_{t \in \R} I(t*\overline{u}) = I(0 *\overline{u}) = I(\overline{u}).
\end{equation}
We now consider the path  $h:[0,1] \to S$ given by $h(t) = (\underline{u}, h_s(t))$ where
$$
h_s(t)(x) =  e^{s(2t-1) \frac{N}{2}}\overline{u}\big(e^{s(2t-1)}x\big).
$$
Here $s>0$ is choosen sufficiently large so that
$$h_s(0)(\cdot)= e^{-s \frac{N}{2}}\overline{u}(e^{-s}\cdot) \in A(\underline{u}),\, \,   h_s(1)(\cdot)= e^{s \frac{N}{2}}\overline{u}(e^s \cdot) \not \in A_{2c(\underline{u})} \, \mbox{ and } \, \max J (\underline{u}, h_s(1)) < 0.$$
Thus $h$ belongs to $\Ga$. Now using (\ref{star6}) and $\beta \geq 0$ we obtain
$$
\max_{t\in[0,1]}J(h(t))
 \le J(\underline{u},0)+\max_{t\in[0,1]}J(0,h_s(t)) = m_{p_1}^{\mu_1}(a_1)+m_{p_2}^{\mu_2}(a_2).
$$
\end{proof}


\begin{lem}\label{lem:PS-exist}
Assume that (H1) and (H2) hold. There exists a Palais-Smale sequence $\{(u_1^n,u_2^n)\}\subset\cS$ for $J$ at the level $\ga(a_1,a_2)$ which satisfies $(u_1^n)^- \to 0$, $(u_2^n)^- \to 0$ in $H$ and the additional property that $Q(u_1^n, u_2^n) \to 0$ where $Q$ is given in \eqref{eq:def-Q}.
\end{lem}

\begin{remark}\label{natural}
It is possible to prove that any solution $(u_1,u_2)$ of \eqref{eq:1.1}, \eqref{eq:1.2} must satisfy $Q(u_1,u_2)=0$. Thus $Q(u_1,u_2)=0$ is a natural constraint. This condition is directly related to the Pohozaev identity adapted to the presence of the constraint $S$. Formally it can be obtained by looking at the function $t \mapsto (t*u_1,t*u_2)$ for $(u_1,u_2) \in \cS.$ Then $Q(u_1,u_2)=0$ corresponds to the condition that the derivative of $t \mapsto J(t*u_1,t*u_2)$ is zero when $t=1$.
\end{remark}

Results in the spirit of Lemma \ref{lem:PS-exist} have now been proved in a variety of situations \cite{BaVa,BeJe,Je,JeLuWa,Lu,HiIkTa} and we shall be rather sketchy here, refering the readers to these papers for more details. We recall the stretched functional first introduced in \cite{Je}:
$$\widetilde{J}: \R \times E \to \R,\quad  (s,(u_1,u_2)) \mapsto J(s*u_1,s*u_2).$$
In the sequel we write
$s*(u_1,u_2):=(s*u_1,s*u_2)$ and recall that $s*(u_1,u_2) \in \cS$ if $(u_1,u_2) \in \cS$. Now we define the set of paths
$$
\begin{aligned}
\widetilde{\Ga}
 &:= \big\{\widetilde{g} \in C([0,1], \R \times \cS):\ \widetilde{g}(0)= (0,(v_1, v_2)), \,
        \widetilde{g}(1)= (0,(w_1, w_2))\\
 &\hspace{2cm}
       v_2\in A(v_1),\,w_2\notin A_{2c(w_1)},\,\max\{J(v_1,v_2),J(w_1,w_2)\}<\inf_BJ\big\}
\end{aligned}
$$
and
$$
\widetilde{\gamma}(a_1,a_2)
:= \inf_{\widetilde{g}\in\widetilde{\Gamma}}\max\limits_{t\in[0,1]}\widetilde{J}(\widetilde{g}(t)).
$$
Observe that $\widetilde{\gamma}(a_1,a_2) = \gamma(a_1,a_2)$. Indeed, by the definitions of $\widetilde{\gamma}(a_1,a_2)$ and $\gamma(a_1,a_2)$ this identity follows immediately from the fact that the maps
$$
\varphi: \Gamma \to \widetilde{\Gamma},\ g \mapsto \varphi(g):=(0,g),
$$
and
$$
\psi:\widetilde{\Gamma}\to\Gamma,\ \widetilde{g}=(\sigma,g)\mapsto\psi(\widetilde{g}):=\sigma*g,\
\text{ with }(\sigma*g)(t)=\sigma(t)*g(t),
$$
satisfy
$$
\widetilde{J}(\varphi(g))=J(g)\ \mbox{ and }\ J(\psi(\widetilde{g}))=\widetilde{J}(\widetilde{g}).
$$

\begin{proof}[Proof of Lemma \ref{lem:PS-exist}]
From the observation that $\widetilde{\gamma}(a_1,a_2) = \gamma(a_1,a_2)$ we obtain a sequence $\{(u_1^n, u_2^n)\} \subset \cS$ such that
$$\max_{t \in [0,1]}\widetilde{J}(0, (v_1^n,v_2^n)) \to \widetilde{\gamma}(a_1,a_2).$$
Since $J(u_1,u_2)=J(|u_1|,|u_2|)$ we can assume that $v_1^n(t) \geq 0$ and $v_2^n(t) \geq 0$ for $t \in [0,1]$. \medskip

Now Ekeland's variational principle implies the existence of a Palais-Smale sequence $\{(s_n, (u_1^n,u_2^n))\}$ for $\widetilde{J}$ restricted to $\R \times \cS$ at the level $\gamma(a_1,a_2)$ such that $s_n \to 0$ and $u_i^n-v_i^n \to 0$ for $i=1,2$. It follows that $(u_1^n)^- \to 0$ and $(u_2^n)^- \to 0$. From $\widetilde{J}(s,(u_1,u_2)) = \widetilde{J}(0, s*(u_1,u_2))$ we deduce that
$$ (\pa_s \widetilde{J})(s,(u_1,u_2)) = (\pa_s \widetilde{J})(0, s* (u_1,u_2))$$
and, for $u=(u_1,u_2)$, $\phi=(\phi_1,\phi_2)$:
$$
(\pa_u \widetilde{J})(s,u)[\phi] = (\pa_u \widetilde{J})(0, s*u)[s*\phi].
$$
As a consequence, $\{(0,s_n *(u_1^n, u_2^n))\}$ is also a Palais-Smale sequence for $\widetilde{J}$ restricted to $\R \times \cS$ at the level $\gamma(a_1,a_2)$. Thus we may assume that $s_n =0$. This implies, firstly, that $\{(u_1^n,u_2^n)\} \subset S$ is a Palais-Smale sequence for $J$ restricted to $S$ at the level $\gamma(a_1,a_2)$ and secondly using $\pa_s \widetilde{J}(0,(u_1^n, u_2^n)) \to 0$ that $Q(u_1^n, u_2^n) \to 0$ holds.
\end{proof}

\begin{lem}\label{lem:PS-bound}
Assume (H1) and (H2) hold. Then the sequence $\{(u_1^n, u_2^n)\} \subset \cS$  obtained in Lemma~\ref{lem:PS-exist} is bounded.
\end{lem}

\begin{proof}
This property is directly related to the fact that the functional $J$ restricted to the set  $Q(u_1,u_2)=0$ is coercive. Indeed we can write, for any $\eps>0$,
\[
\begin{aligned}
J(u_1,u_2)
 &= \frac{\eps}{2} |\nabla u_1|_2^2 + \frac{\eps}{2} |\nabla u_2|_2^2
     +  a(\eps) |u_1|_{p_1}^{p_1} + b(\eps) |u_2|_{p_2}^{p_2} \\
 &\hspace{1cm}
     + \be c(\eps) \int_{\R^N}|u_1|^{r_1}|u_2|^{r_2} \, dx
     + \frac{1-\eps}{2}Q(u_1,u_2).
\end{aligned}
\]
where
$$
a(\eps) = \frac{(1 - \eps)\mu_1N}{2 p_1}\left(\frac{p_1}{2}-1\right) - \frac{\mu_1}{p_1}, \quad b(\eps) = \frac{(1 - \eps)\mu_2N}{2 p_2} \left(\frac{p_2}{2}-1\right) - \frac{\mu_2}{p_2}
$$
and
$$
c(\eps) = \frac{(1 - \eps)N}{2} \left(\frac{r_1 + r_2}{2} -1\right)-1.
$$
The coefficient $a(\eps)$ is strictly negative but the corresponding term can be easily controlled by $\eps|\nabla u_1|_2^2$ using the Gagliardo-Nirenberg inequality once more because $p_1<2+\frac4N$. Next observe that $b(\eps)>0$ holds for $\eps>0$ small enough, because $p_2 > 2+\frac4N$. Now concerning the term $\be c(\eps) \int_{\R^N}|u_1|^{r_1}|u_2|^{r_2} \, dx$ we immediately obtain that $c(\eps) > 0$ for $\eps>0$ small. Using $Q(u_1^n,u_2^n) \to 0$  yields the boundedness of our Palais-Smale sequence.
\end{proof}

At this point, using Lemma \ref{Palais-Smale} we can assume that $(u_1^n,u_2^n) \rightharpoonup (u_1,u_2)$ weakly in $E$ with $u_1 \geq 0$ and $u_2 \geq 0$. In order to get the strong convergence, according to Lemmas~\ref{Palais-Smale} and \ref{Step5}, we just need to show that $u_1\ne0$ and $u_2\ne0$.

\begin{lem}\label{lem:ultimate}
Assume that (H1) and (H2) hold, and that $\ga(a_1,a_2)\ne0$. Then $u_1\ne0$ and $u_2\ne0$.
\end{lem}

\begin{proof}
Suppose by contradiction that at least one of $u_1$ or $u_2$ is zero. Then the strong convergence in $L^q(\R^N)$ for $q\in(2,2^*)$ implies
$$
\be \int_{\R^N} |u_1^n|^{r_1}|u_2^n|^{r_2} \, dx \to 0.
$$
Thus since $\{(u_1^n, u_2^n)\}$  satisfies $Q(u_1^n, u_2^n)  \to 0$ it follows that
\begin{equation}\label{underweak}
\begin{aligned}
J(u_1^n,u_2^n)
 &= \frac{\mu_1}{p_1}\left[\frac{N}{2}\left(\frac{p_1}{2}-1\right) -1\right] |u_1^n|_{p_1}^{p_1}
    + \frac{\mu_2}{p_2}\left[\frac{N}{2}\left(\frac{p_2}{2}-1\right)-1\right]|u_2^n|_{p_2}^{p_2} \\
 &= -D_1 |u_1|_{p_1}^{p_1} + D_2 |u_2|_{p_2}^{p_2} + o(1).
\end{aligned}
\end{equation}
where $D_1 >0$ and $D_2 >0$.  We now distinguish three cases: \medskip

\noindent{\it Case 1 : $u_1 =0$ and $u_2=0$.} \\
From (\ref{underweak}) we obtain that $J(u_1^n, u_2^n) \to 0$. Thus since we have assumed that $\gamma(a_1,a_2) \ne 0$ this case cannot occur. \medskip

\noindent{\it Case 2 : $u_1 =0$ and $u_2 \neq 0$.} \\
First note that by Lemma \ref{Step5}  we have $\la_2 <0$, hence $u_2^n \to u_2 \in S(a_2)$ strongly in $H$ as a consequence of Lemma~\ref{Palais-Smale}. Using $u_1=0$ it follows from (\ref{underweak}) that
\beq[1.11]
J(u_1^n,u_2^n) - I (u_2^n) \to 0 \quad \mbox{and} \quad J(u_1^n,u_2^n) \to D_2|u_2|_{p_2}^{p_2}.
\eeq
Since $(u_1,u_2)$ is a solution of the system \eqref{eq:1.1} we see that $u_2 \gneqq 0$ satisfies
$$- \De u - \la_2 u = \mu_2 |u|^{p_2 -2}u.$$
From Lemma \ref{lem-ground} and \eqref{eq:1.11} we deduce that $D_2|u_2|_{p_2}^{p_2} = m_{p_2}^{\mu_2}(a_2) >0$. Therefore in order to obtain a contradiction it suffices to show that $\gamma(a_1,a_2) < m_{p_2}^{\mu_2}(a_2).$  But this is immediate from Lemma \ref{upperbound} because $m_{p_1}^{\mu_1}(a_1)<0$. Thus case 2 is not possible. \medskip

\noindent{\it Case 3 : $u_1  \neq 0$ and $u_2 = 0.$} \\
As in case 2 we can show that $u_n^1 \to u_1 \in S(a_1) $ strongly in $H$. Now since $u_2 = 0$ it follows that
$$
J(u_1^n, u_2^n) - I(u_1^n)  \to 0 \quad \mbox{and} \quad J(u_1^n,u_2^n) \to - D_1 |u_1|_{p_1}^{p_1}.
$$
Arguing as in case 2 we identify $- D_1 |u_1|_{p_1}^{p_1}$ with the least energy level of
\begin{equation}\label{onedim1}
- \De u - \la_1 u =  \mu_1 |u|^{p_1 -2}u,
\end{equation}
namely
$$
- D_1|u_1|_{p_1}^{p_1} = m_{p_1}^{\mu_1}(a_1).
$$
Therefore in order to avoid that this case happens it suffices to show that
$\ga(a_1,a_2) > m_{p_1}^{\mu_1}(a_1) = I(\underline{u})$. But this is precisely what we can deduce from the lemmas~\ref{lem:def-c}, \ref{intersection} and the definition of $B$. \medskip

Having proved that the cases 1, 2 and 3 are both impossible this concludes the proof of the lemma.
\end{proof}

\begin{proof}[Proof of Theorem \ref{th:main}]
In view of the lemmas~\ref{Palais-Smale}, \ref{lem:PS-exist}, \ref{lem:PS-bound} and \ref{lem:ultimate}, in order  to establish the theorem it is enough to prove that $\ga(a_1,a_2) <0$. We see from Lemma~\ref{upperbound} that this is the case if $m_{p_1}^{\mu_1}(a_1) + m_{p_2}^{\mu_2}(a_2) <0.$ Note also that $u_1 >0$ and $u_2 >0$ follows directly from the strong maximum principle because $u_1\gneqq0$ and $u_2\gneqq0$.
\end{proof}

\begin{proof}[Proof of Corollary \ref{cor:main}]
From Lemma \ref{lem-ground} we know that $m_{p_1}^{\mu_1}(a_1) <0$ and $m_{p_1}^{\mu_1}(a_1) \to - \infty$ as $a_1 \to 0$. Also $m_{p_2}^{\mu_2}(a_2) >0$ and $m_{p_2}^{\mu_2}(a_2) \to 0$ as $a_2 \to \infty$. Therefore the corollary follows directly from Theorem \ref{th:main}.
\end{proof}

{\sc Address of the authors:}\\[1em]
\begin{tabular}{ll}
 Thomas Bartsch & Louis Jeanjean\\
 Mathematisches Institut & Laboratoire de Math\'ematiques (UMR 6623)\\
 Universit\"at Giessen & Universit\'{e} de Franche-Comt\'{e}\\
 Arndtstr.\ 2 & 16, Route de Gray\\
 35392 Giessen & 25030 Besan\c{c}on Cedex\\
 Germany & France\\
 Thomas.Bartsch@math.uni-giessen.de & louis.jeanjean@univ-fcomte.fr
\end{tabular}

\end{document}